\documentclass[11pt]{article}
\usepackage{graphicx}
\usepackage{amssymb,amsmath}
\usepackage[english]{babel}
\usepackage{enumerate}
\usepackage{enumerate}

\newtheorem{DE}{Definition}[section]

\newcommand {\sm} {\setminus}

\usepackage{latexsym} 
\usepackage{theorem} 
\newcommand{\qed}{\relax\ifmmode\hskip2em\Box\else\unskip\nobreak\hfill$\Box$\fi}

\newtheorem{theorem}[DE]{Theorem}
\newtheorem{lemma}[DE]{Lemma}
\newtheorem{conjecture}[DE]{Conjecture}
\newtheorem{corollary}[DE]{Corollary}
{\theoremstyle{break}\theorembodyfont{\rmfamily}}
{\theoremstyle{break}\theorembodyfont{\rmfamily}}

\newcounter{claim}
\newenvironment{proof}[1][]%
	{\noindent {\setcounter{claim}{0}\sc proof --- }{#1}{}}{\qed\vspace{2ex}}
	{\refstepcounter{claim}\vspace{1ex}\noindent {(\it\arabic{claim}) {#1}{}}\it}{\vspace{1ex}}
	{\noindent {}{#1}{}}{ This proves~(\arabic{claim}).\vspace{1ex}}

\begin{document}

\title{Linear balanceable and subcubic balanceable graphs}

\author{Pierre Aboulker\thanks{Universit\'e Paris 7, LIAFA, Case 7014,
    75205 Paris Cedex 13, France.  E-mail:
    aboulker@liafa.jussieu.fr.} ,
Marko Radovanovi\'c\thanks{Faculty of Computer Science (RAF), Union
    University, Knez Mihailova 6/VI, 11000 Belgrade, Serbia. E-mail:
mradovanovic@raf.edu.rs. Partially supported by
  Serbian Ministry
    of Education and Science project 174033.} , Nicolas Trotignon\thanks{CNRS, LIP, ENS
Lyon, INRIA, Universit\'e de Lyon, 15 parvis Ren\'e Descartes BP 7000
69342 Lyon cedex 07 France. E-mail: nicolas.trotignon@ens-lyon.fr.} ,\\Th\'eophile Trunck\thanks{ENS
Lyon, LIP, CNRS, INRIA, Universit\'e de Lyon, 15 parvis Ren\'e Descartes BP 7000
69342 Lyon cedex 07 France. E-mail: theophile.trunck@ens-lyon.fr.}
 ~and Kristina
  Vu\v{s}kovi\'c\thanks{School of Computing, University of Leeds,
    Leeds LS2 9JT, UK and
Faculty of Computer Science (RAF), Union
    University, Knez Mihailova 6/VI, 11000 Belgrade, Serbia. E-mail:
    k.vuskovic@leeds.ac.uk.  Partially supported by
 EPSRC grant EP/H021426/1
and Serbian Ministry
    of Education and Science projects 174033 and III44006.
\newline The first four authors are partially supported by \emph{Agence Nationale de la Recherche} under reference
    \textsc{anr 10 jcjc 0204 01}. All five authors are partially supported by PHC Pavle Savi\'c grant 2010-2011,
jointly awarded by EGIDE,
an agency of the French Minist\`ere des Affaires \'etrang\`eres et
europ\'eennes, and Serbian
Ministry of Education and Science.}}

\date{December 21, 2012}

\maketitle

\begin{abstract}
  In [{Structural properties and decomposition of linear balanced
    matrices}, {\it Mathematical Programming}, 55:129--168, 1992],
  Conforti and Rao conjectured that every balanced bipartite graph
  contains an edge that is not the unique chord of a cycle. We prove
  this conjecture for balanced bipartite graphs that do not contain a
  cycle of length 4 (also known as linear balanced bipartite graphs),
  and for balanced bipartite graphs whose maximum degree is at most
  3. We in fact obtain results for more general classes, namely linear
  balanceable and subcubic balanceable graphs.  Additionally, we prove
  that cubic balanced graphs contain a pair of twins, a result that was
  conjectured by Morris, Spiga and Webb in [Balanced Cayley graphs and
  balanced planar graphs, {\it Discrete Mathematics}, 310:3228--3235,
  2010].
\end{abstract}

\section{Introduction}

A 0,\,1 matrix is {\em balanced} if for every square submatrix with
two ones per row and column, the number of ones is a multiple of
four. This notion was introduced by Berge~\cite{berge-bal}, and later
extended to $0$,\,$\pm1$ matrices by Truemper~\cite{truemper-alpha_bal}.  A 0,\,$\pm1$ matrix is {\em balanced} if
for every square submatrix with two nonzero entries per row and
column, the sum of the entries is a multiple of four.  These matrices
have been studied extensively in literature due to their important
polyhedral properties, for a survey see~\cite{ccv-survey}.

Given a 0,\,1 matrix $A$, the {\em bipartite graph representation of}
$A$ is the bipartite graph having a vertex for every row in $A$, a
vertex for every column of $A$, and an edge $ij$ joining row $i$ to
column $j$ if and only if the entry $a_{ij}$ of $A$ equals 1.  We say
that $G$ is {\em balanced} if it is the bipartite graph representation of some
balanced matrix.  It is easy to see that a bipartite graph $G$ is
balanced if and only if every hole of $G$ has length 0\,(mod\,4),
where a {\em hole} is a chordless cycle of length at least 4.  A {\em
  signed bipartite graph} is a bipartite graph, together with an
assignment of weights $+1$,\,$-1$ to the edges of $G$. A signed
bipartite graph is {\em balanced} if the {\em weight} of every hole
$H$ of $G$, i.e.\ the sum of the weights of the edges of $H$, is
0\,(mod\,4). A bipartite graph is {\em balanceable} if there exists a
{\em signing} of its edges, i.e.\ an assignment of weights $+1$,\,$-1$
to the edges of the graphs, such that the resulting signed bipartite
graph is balanced.

The following conjecture is the last unresolved conjecture about
balanced (balanceable) bipartite graphs in Cornu\'ejols'
book~\cite{cornuejols:packing} (it is Conjecture~6.11). Note that
Conjectures 9.23, 9.28 and 9.29 from~\cite{cornuejols:packing} have
been resolved by Chudnovsky and Seymour
in~\cite{chudnovsky.seymour:cornu}.

\begin{conjecture}[Conforti and Rao~\cite{cr-lin}]\label{cr-conj}
Every balanced bipartite graph contains an edge that is not the unique
chord of a cycle.
\end{conjecture}

In other words, every balanced bipartite graph contains an edge whose
removal leaves the graph balanced.  This is not true if the graph is
balanceable, as shown by $R_{10}$, that is the graph defined by the
cycle $x_1x_2\ldots x_{10}x_1$ (of length 10) with chords
$x_ix_{i+5}$, $1\leq i\leq 5$ (see Figure \ref{fig:R_10} -- in all
figures in this paper a solid line denotes an edge, and a dashed one a
path of length greater than 1).  Graph $R_{10}$ is cubic and
balanceable (a proper signing of $R_{10}$ is to assign weight $+1$ to
the edges of the cycle $x_1x_2\ldots x_{10}x_1$ and $-1$ to the
chords), but not balanced ($x_1x_2x_3x_4x_5x_6$ is a hole of length
6).  Note that in $R_{10}$ every edge is the unique chord of some
cycle. Conjecture~\ref{cr-conj} generalises to balanceable graphs in
the following way.

\begin{figure}[hbtp]
\begin{center}
\includegraphics{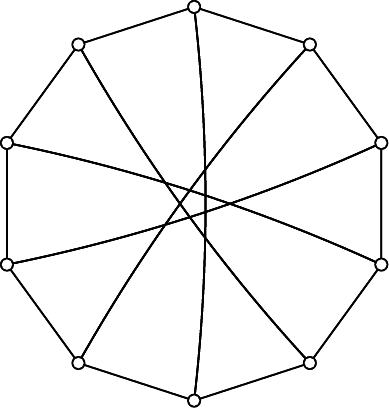}\hspace{2em}
\includegraphics{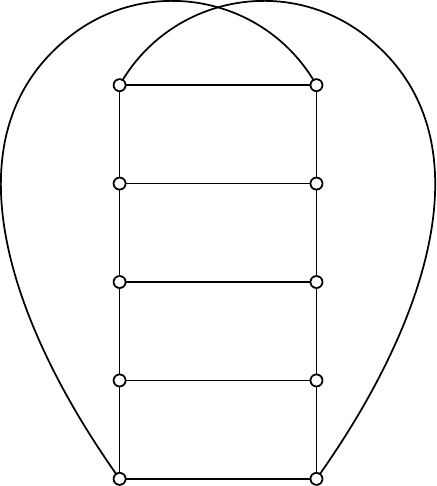}
\end{center}
\caption{Two ways to draw the graph $R_{10}$.\label{fig:R_10}}
\end{figure}

\begin{conjecture}[Conforti, Cornu\'ejols and Vu\v skovi\'c~\cite{ccv-survey}]\label{conj}
  In a balanceable bipartite graph either every edge belongs to some
  $R_{10}$ or there is an edge that is not the unique chord of a
  cycle.
\end{conjecture}

These conjectures are known to be true for several classes of graphs.
A bipartite graph is {\em restricted balanceable} if there exists a
signing of its edges so that in the resulting signed graph every cycle
(induced or not) is balanced. Clearly no edge of a restricted
balanceable bipartite graph can be the unique chord of a cycle. In
other words, the removal of any subset of edges from a restricted
balanceable graph leaves the graph restricted balanceable.  A
bipartite graph is {\em strongly balanceable} if it is balanceable and
does not contain a cycle with a unique chord.  Figure~\ref{fig:not
  strongly bal} shows that there are cubic balanceable graphs that are
not strongly balanceable.  This class generalizes restricted
balanceable graphs, and it clearly satisfies Conjecture~\ref{conj}.  On
the other hand, removing any edge from a strongly balanceable graph
might not leave the graph strongly balanceable. In~\cite{cr-stbal} it
is shown that every strongly balanceable graph has an edge whose
removal leaves the graph strongly balanceable.  A bipartite graph is
{\em totally balanced} if every hole of $G$ is of length 4. It is
shown in~\cite{gg} that every totally balanced bipartite graph has a
bisimplicial edge (i.e.\ an edge $uv$ such that the node set $N(u)
\cup N(v)$ induces a complete bipartite graph). So clearly, the graph
obtained by removing a bisimplicial edge from a totally balanced
bipartite graph is also totally balanced.

\begin{figure}[hbtp]
\begin{center}
\includegraphics{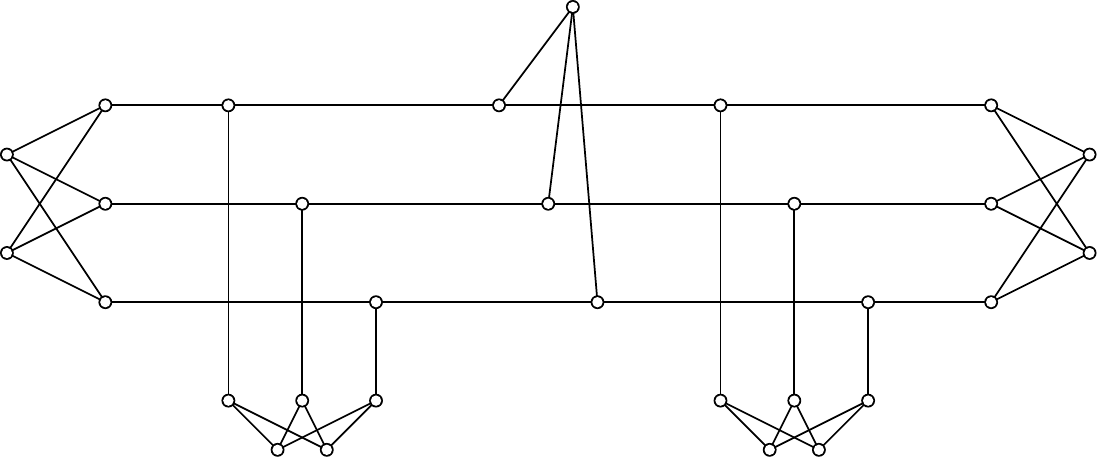}
\end{center}
\caption{Cubic balanceable graph that is not strongly balanceable.\label{fig:not strongly bal}}
\end{figure}

A bipartite graph is {\em linear balanceable} if it is balanceable and
does not contain a 4-hole (i.e.\ a hole of length 4). A graph $G$ is
{\em subcubic} if $\Delta(G) \le 3$.  In this paper, we prove
that conjectures~\ref{cr-conj} and~\ref{conj} hold when restricted to
linear balanceable graphs (see Corollary~\ref{cr-square}) and to
subcubic balanceable graphs (see Corollary~\ref{cr-cubic}).  For the subcubic
case, our proof relies on a result conjectured by~Morris, Spiga and
Webb~\cite{spiga}, stating that every cubic balanced bipartite graph
contains a pair of vertices with the same neighborhood (see
Corollary~\ref{twins}).

Our proofs are based on known decomposition theorems for the classes
we consider, which we describe in Section~\ref{decomposition}. The
decomposition theorems say that either the graph belongs to some
simple subclass, that we call basic, or it has a 2-join, 6-join or
star cutset.  It is not straightforward to use these decomposition
theorems to prove the desired results.  In fact, the decomposition
theorem for balanced bipartite graphs~\cite{ccr-bal} has been known
since the early 1990's, and still no one knows how to use it to prove
the Conforti and Rao Conjecture.  The key idea that makes things work
for us, is the use of extreme decompositions, i.e.\ decompositions in
which one of the blocks is basic. In Section~\ref{secnostar} we prove
that if star cutsets are excluded, then the graphs in our classes
admit extreme decompositions. This is sufficient for the proof of the
main result in the subcubic case in Section~\ref{cubic}, since the
induction hypothesis in this case goes through the star cutset
nicely. For the linear balanceable bipartite graphs, this is not the
case. Here we cannot inductively get rid of star cutsets in a
straightforward manner.  Furthermore, it is not true that if a (linear
balanceable) graph has a star cutset, then it has a star cutset one of
whose blocks of decomposition does not have a star cutset. Instead, to
prove the main result for linear balanceable graphs in
Section~\ref{squarefree}, we develop a new technique for finding an
``extreme decomposition'' with respect to star cutsets: we look for a
minimally-sided double star cutset, and show that the corresponding
block of decomposition does not have a star cutset.

\subsection*{Terminology}

We say that a graph $G$ {\em contains} a graph $H$ if $H$ is
isomorphic to an induced subgraph of $G$. A graph $G$ is {\em
  $H$-free} if it does not contain $H$. For $x\in V(G)$, $N(x)$
denotes the set of neighbors of~$x$.  For $S \subseteq V(G)$, $G[S]$
denotes the subgraph of $G$ induced by $S$, and $G\setminus
S=G[V(G)\setminus S]$.  For $S\subseteq E(G)$, $G\setminus S$ denotes
the graph obtained from $G$ by deleting edges from $S$.

A {\em path} $P$ is a sequence of distinct vertices $p_1p_2\ldots
p_k$, $k\geq 1$, such that $p_ip_{i+1}$ is an edge for all $1\leq i
<k$.  Edges $p_ip_{i+1}$, for $1\leq i <k$, are called the {\em edges
  of $P$}.  Vertices $p_1$ and $p_k$ are the {\em ends} of $P$.  A
cycle $C$ is a sequence of vertices $p_1p_2\ldots p_kp_1$, $k \geq 3$,
such that $p_1\ldots p_k$ is a path and $p_1p_k$ is an edge.  Edges
$p_ip_{i+1}$, for $1\leq i <k$, and edge $p_1p_k$ are called the {\em
  edges of $C$}.  Let $Q$ be a path or a cycle.  The vertex set of $Q$
is denoted by $V(Q)$.  The {\em length} of $Q$ is the number of its
edges.  An edge $e=uv$ is a {\em chord} of $Q$ if $u,v\in V(Q)$, but
$uv$ is not an edge of $Q$. A path or a cycle $Q$ in a graph $G$ is
{\em chordless} if no edge of $G$ is a chord of $Q$.  The {\em girth} of a
graph is the length of its shortest cycle.

A {\em cut vertex} of a connected graph $G$ is a vertex $v$ such that
$G\setminus\{v\}$ is disconnected. A {\em block} of a graph is a connected
subgraph that has no cut vertex and that is maximal with respect to
this property.  We may associate with any graph $G$ a graph $B(G)$ on
${\mathcal B} \cup S$, where $\mathcal B$ is the set of blocks of $G$
and $S$ the set of cut vertices of $G$, a block $B$ and a cut vertex
$v$ being adjacent if and only if $B$ contains $v$.  It is a classical
result that $B(G)$ is a tree (see~\cite{bondy.murty:book}). The blocks
that correspond to leaves of $B(G)$ are the {\em end blocks} of $G$.

\section{Decomposition theorems}\label{decomposition}

In this section we describe known decomposition theorems for
balanceable graphs.  First, we state the forbidden induced subgraph characterization of balanceable graphs.  Let $G$ be a bipartite graph. Let $u$, $v$ be two nonadjacent
vertices of $G$. A {\em 3-path configuration connecting $u$ and $v$},
is defined by three chordless paths $P_1$, $P_2$, $P_3$ with ends $u$
and $v$, such that the vertex set $V(P_i)\cup V(P_j)$ induces a hole,
for $i,j\in\{1,2,3\}$ and $i\neq j$. A 3-path configuration is said to
be {\em odd} if it connects two vertices that are on opposite sides of
the bipartition.  A {\em wheel} is defined by a hole $H$ and a vertex
$x\not\in V(H)$ having at least three neighbors in $H$, say
$x_1,x_2,\ldots,x_n$. If $n$ is even, then the wheel is an {\em even
  wheel}, and otherwise it is an {\em odd wheel}. A 3-path
configuration and an odd wheel are shown in Figure
\ref{fig:3pc+oddwheel}.

\begin{figure}[hbtp]
\begin{center}
\includegraphics{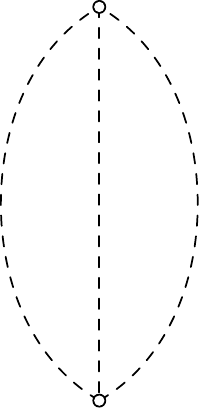}\qquad\qquad
\includegraphics{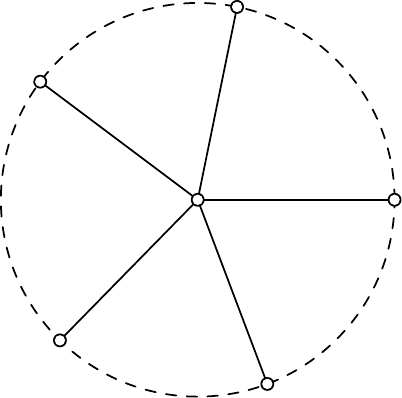}
\end{center}
\caption{3-path configuration and an odd wheel.\label{fig:3pc+oddwheel}}
\end{figure}

It is easy to see that a balanceable graph does not contain an odd
3-path configuration, nor an odd wheel. The following theorem of
Truemper states that the converse is also true.

\begin{theorem}[Truemper~\cite{truemper-alpha_bal}]\label{balanceable-truemper}
  A bipartite graph is balanceable if and only if it does not contain
  an odd wheel nor an odd 3-path configuration.
\end{theorem}

Now, we introduce different cutsets used in the decomposition theorems
that we need.


A  set $S$ of vertices (resp.\ edges) of a connected graph  $G$ is a {\em vertex cutset} (resp.\ {\em edge cutset}) if
the subgraph $G \setminus S$ is disconnected.

\vspace{2mm}

\noindent{\bf 1-join}
\\
A graph $G$ has a {\em 1-join} if $V(G)$ can be partitioned into sets $X$ and $Y$ so that the following hold:
\begin{itemize}

\item $|X|\geq 2$ and $|Y|\geq 2$.

\item There exist sets $A$ and $B$ such that
$\emptyset\neq A\subseteq X$ and $\emptyset\neq B\subseteq Y$;
there are all possible edges between $A$ and $B$; and
there are no other edges between $X$ and $Y$.

\end{itemize}
We say that $(X,Y,A,B)$ is a {\em split} of this 1-join.

\vspace{2ex}

\noindent{\bf 2-join}
\\
A graph $G$ has a {\em 2-join} $(X_1,X_2)$ if $V(G)$ can be partitioned into sets $X_1$ and $X_2$ so
that the following hold:
\begin{itemize}
\item For $i=1,2$, $X_i$ contains disjoint nonempty sets $A_i$ and $B_i$, such that
every vertex of $A_1$ is adjacent to every vertex of $A_2$,
every vertex of $B_1$ is adjacent to every vertex of $B_2$, and there are no other adjacencies
between $X_1$ and $X_2$.
\item For $i=1,2$, $X_i$ contains at least one path from $A_i$ to $B_i$, and if
$|A_i|=|B_i|=1$, then $G[X_i]$ is not a chordless path.
\end{itemize}

We say that $(X_1,X_2,A_1,A_2,B_1,B_2)$ is a {\em split} of this 2-join, and the sets $A_1,A_2,B_1,B_2$ are the {\em special sets} of this 2-join.

\vspace{2ex}

\noindent{\bf 6-join}
\\
A graph $G$ has a {\em 6-join} $(X_1,X_2)$ if $V(G)$ can be partitioned into sets $X_1$ and $X_2$ so that the following hold:
\begin{itemize}

\item $X_1$ (resp.\ $X_2$) contains disjoint nonempty sets $A_1, A_3,
  A_5$ (resp.\ $A_2, A_4, A_6$) such that, for every $i \in \{ 1,
  \dots ,6\}$, every vertex in $A_i$ is adjacent to every vertex in
  $A_{i-1}\cup A_{i+1}$ (where subscripts are taken modulo 6), and
  these are the only adjacencies between $X_1$ and $X_2$.

\item $|X_1|\geq 4$ and $|X_2|\geq 4$.

\end{itemize}

We say that $(X_1,X_2,A_1,A_2,A_3,A_4,A_5,A_6)$ is a {\em split} of this 6-join.

\vspace{2ex}

\noindent{\bf Extended star cutset}
\\
In a connected bipartite graph $G$, $(x,T,A,R)$ is an {\em extended
  star cutset} if $T$, $A$, $R$ are disjoint subsets of $V(G)$, $x\in
T$ and the following hold:
\begin{itemize}

\item The graph $G \setminus (T\cup A \cup R)$ is disconnected.

\item $A\cup R\subseteq N(x)$

\item The vertex set $T\cup A$ induces a complete bipartite graph (with vertex set $T$ on one side of the bipartition and vertex set $A$ on the other).

\item If $|T|\geq 2$, then $|A|\geq 2$.

\end{itemize}

An extended star cutset such that $T=\{x\}$ is a {\em star cutset}. In
this paper we will denote it as $(x,R)$.  Note that when $|T|=1$ and
$A\cup R=\emptyset$ then $\{ x \}$ is a cut vertex.

\medskip

The following theorem is proved in~\cite{cckv-bal}, building on the
decomposition theorem in~\cite{ccr-bal}. We observe that the
definition of 2-join in~\cite{ccr-bal} and~\cite{cckv-bal} is slightly
different from the one we gave here. We define the 2-join and state
the following theorem as in~\cite{ccv-survey}.  The statement is
easily seen to be equivalent to the one in~\cite{cckv-bal} by Lemma
\ref{l1} below.

\begin{theorem}[Conforti, Cornu\'ejols, Kapoor and Vu\v{s}kovi\'c~\cite{cckv-bal}]\label{cckv-thm}
  A connected balanceable bipartite graph is either strongly
  balanceable or is $R_{10}$, or it has a 2-join, a 6-join or an
  extended star cutset.
\end{theorem}

\begin{theorem}[Conforti and Rao~\cite{cr-stbal}]\label{cr-thm}
  A strongly balanceable bipartite graph is either restricted
  balanceable or has a 1-join.
\end{theorem}

A bipartite graph is {\em basic} if it admits a bipartition such that all the vertices in one side of the bipartition have degree
at most 2.

\begin{theorem}[Yannakakis~\cite{yannakakis}]\label{yan}
  A restricted balanceable bipartite graph is either basic or has a
  cut vertex or a 2-join whose special sets are all of size 1.
\end{theorem}

The following lemma is proved in~\cite{nicolas.kristina:two} (Lemma
3.2) and a special case of it is proved in~\cite{ccr-bal} (Lemma 2.4).

\begin{lemma}\label{l1}
Let $G$ be a graph that has no star cutset, and let
$(X_1,X_2,A_1,A_2,B_1,B_2)$ be a split of a 2-join of $G$.
Then for $i=1,2$, the following hold:
\begin{enumerate}[(i)]
\item Every component of $G[X_i]$ meets both $A_i$ and $B_i$.
\item Every $u\in X_i$ has a neighbor in $X_i$.
\item Every vertex of $A_i$ has a non-neighbor in $B_i$.
\item Every vertex of $B_i$ has a non-neighbor in $A_i$.
\item $|X_i|\geq 4$.
\end{enumerate}
\end{lemma}

\begin{lemma}\label{l1'}
Let $G$ be a bipartite graph that has no star cutset. If $G$ has a 1-join, then $G$ is a 4-hole.
\end{lemma}
\begin{proof}
  Let $(X,Y,A,B)$ be a split of a 1-join of $G$.  If $Y\setminus
  B\neq\emptyset$, then a vertex from $A$ and set $B$ form a star
  cutset, a contradiction. So $Y=B$, and by symmetry $X=A$. If
  $|A|\geq 3$, then a vertex from $A$ and set $B$ form a star cutset,
  a contradiction. So, by symmetry, $|A|=|B|=2$, and therefore $G$ is
  a 4-hole.
\end{proof}

In 4-hole-free graphs, and also in subcubic graphs, we can reduce
extended star cutset to star cutset. Indeed in a 4-hole-free graph, if
$(x,T,A,R)$ is an extended star cutset with $|T|\geq 2$, then by
definition $|A|\geq 2$ and the complete bipartite graph $A\cup T$
contains a 4-hole, so $|T|=1$ and $(x,T,A,R)$ is a star cutset. In a
subcubic graph $G$ if $(x,T,A,R)$ is an extended star cutset with
$|T|\geq 2$, then $|A|\geq 2$. Since each vertex of $T$ has neighbors
in at most one component of $G\setminus(T\cup A\cup R)$ (because the
graph is subcubic), we see that $G\setminus (\{x\}\cup A\cup R)$ has
at least as many components as $G\setminus(T\cup A\cup R)$. It follows
that $(x,R\cup A)$ is a star cutset of $G$.  So from
Theorems~\ref{cckv-thm}, \ref{cr-thm} and~\ref{yan}, and Lemmas
\ref{l1} and~\ref{l1'}, we get the following decomposition theorem
that we will use in this paper.

\begin{theorem}\label{dt}
Let $G$ be a connected balanceable bipartite graph.
\begin{itemize}
\item If $G$ is 4-hole-free, then $G$ is basic, or has a
  2-join, a 6-join or a star cutset.
\item If $\Delta (G) \leq 3$, then $G$ is basic or is $R_{10}$, or
has a 2-join, a 6-join or a star cutset.
\end{itemize}
\end{theorem}

We observe that a balanceable bipartite graph $G$ with $\Delta (G)
\leq 3$ is actually \emph{matrix-regular}, as we explain now.  A matrix is
{\em totally unimodular} if every square submatrix has determinant
equal to $0$, $+1$ or $-1$. A $0,\,1$ matrix is {\em regular} if its
nonzero entries can be signed $+1$ or $-1$ so that the resulting
matrix is totally unimodular.  A $0,\,1$ matrix $A$ can be thought of
as a vertex-vertex incidence matrix of a bipartite graph, which we
denote with $G(A)$.  We say that a bipartite graph $G$ is {\em
  matrix-regular} if $G=G(A)$ for some regular $0,\,1$ matrix $A$.  A graph
is {\em eulerian} if all its vertices have even degree. By a theorem
of Camion~\cite{camion}, a bipartite graph is matrix-regular if and only if
there exists a signing of its edges with $+1$ or $-1$ so that the
weight of every induced eulerian subgraph is a multiple of~4.  It now
clearly follows that for a bipartite graph $G$ with $\Delta (G)\leq
3$: $G$ is balanceable if and only if $G$ is matrix-regular.

It is natural to ask why we use Theorem~\ref{dt} in our proof of
Conforti and Rao Conjecture in the subcubic case, instead of Seymour's
decomposition theorem for matrix-regular bipartite graphs~\cite{seymour}.
The answer is that by using Theorem~\ref{dt} we have only to check
whether three cutsets (2-join, 6-join and star cutset) go through our
induction hypothesis, whereas if we used the decomposition theorem in~\cite{seymour} we would have to check five cutsets (1-join, 2-join,
6-join, N-join and M-join, for an explanation
see~\cite{vuskovic}). Furthermore, 2-joins and 6-joins in graphs with
no star cutset have special properties (given in Section
\ref{secnostar}) which are very useful for pushing the induction
hypothesis through them.

\section{Graphs with no star cutset}\label{secnostar}

The following properties of graphs with no star cutsets will be
essential in our proofs.

Let $(X_1,X_2,A_1,A_2,B_1,B_2)$ be a split of a 2-join of a graph $G$.
The {\em blocks of decomposition} of $G$ by this 2-join are graphs $G_1$ and $G_2$ defined as follows. To obtain $G_i$, for $i=1,2$, we start from $G[X_i]$, and first add a vertex $a_{3-i}$, adjacent to all the vertices in $A_i$ and no other vertex of $X_i$, and a vertex $b_{3-i}$ adjacent to all the vertices in $B_i$ and no other vertex of $X_i$. For $i=1,2$, let $Q_{3-i}$ be a path in $G[X_{3-i}]$  with smallest number of edges connecting a vertex in $A_{3-i}$ to a vertex in $B_{3-i}$. For $i=1,2$, add to $G_i$ a {\em marker path} $M_{3-i}$ connecting $a_{3-i}$ and $b_{3-i}$ with length $|E(M_{3-i})|\in\{4,5\}$
having the same parity as $Q_{3-i}$.

The following lemma is proved in~\cite{cckv-balalg}. (Note that the
statement is not the same but the proof of Theorem 4.6 in~\cite{cckv-balalg} shows precisely what we need).

\begin{lemma}\label{l:2j}
  Let $G$ be a bipartite graph with no star cutset. Let $(X_1,X_2)$ be
  a 2-join of $G$, and let $G_1$ and $G_2$ be the corresponding blocks
  of decomposition. Then the following hold:
\begin{enumerate}[(i)]
\item If $G$ is balanceable, then $G_1$ and $G_2$ are balanceable.
\item $G_1$ and $G_2$ have no star cutset.
\item If $G$ has no 6-join, then $G_1$ and $G_2$ have no 6-join.
\end{enumerate}
\end{lemma}


Let $(X_1,X_2,A_1, \ldots ,A_6)$ be a split of a 6-join of a graph $G$.
The {\em blocks of decomposition} of $G$ by this 6-join are graphs $G_1$ and $G_2$ defined as follows. For $i=1,\ldots,6$ let $a_i$ be any vertex of $A_i$. Then $G_1=G[X_1\cup\{a_2,a_4,a_6\}]$ and $G_2=G[X_2\cup\{a_1,a_3,a_5\}]$.
Nodes $a_2,a_4,a_6$ (resp.\ $a_1,a_3,a_5$) are called the {\em marker nodes} of $G_1$ (resp.\ $G_2$).

\begin{lemma}\label{6jl1}
  Let $G$ be a bipartite graph with no star cutset. Let $(X_1,X_2,A_1,
  \ldots ,A_6)$ be a split of a 6-join of $G$, and $G_1$ and $G_2$ the
  corresponding blocks of decomposition. Then the following hold:
\begin{enumerate}[(i)]
\item $X_1 \setminus (A_1 \cup A_3 \cup A_5 )\neq \emptyset$ and $X_2
  \setminus (A_2 \cup A_4 \cup A_6 )\neq \emptyset$.
\item If $C$ is a connected component of $G[X_1 \setminus (A_1 \cup
  A_3 \cup A_5)]$ (resp.\ $G[X_2 \setminus (A_2 \cup A_4 \cup A_6)]$),
  then a node of $A_i$, for every $i=1,3,5$ (resp.\ $i=2,4,6$) has a
  neighbor in $C$.
\item If $G$ is 4-hole-free or $\Delta (G) \leq 3$, then $|A_i|=1$ for
  every $i\in\{ 1, \ldots ,6\}$, and in particular every node of
  $\cup_{i=1}^{6} A_i$ is of degree at least 3 in $G$.
\item If $G$ is balanceable, then so are $G_1$ and $G_2$.
\item If $G$ is 4-hole-free, then $G_1$ and $G_2$ do not have star
  cutsets.
\end{enumerate}
\end{lemma}

\begin{proof}
Note that $G$ is bipartite so there are no edges in
$A_1\cup A_3\cup A_5$ nor in $A_2\cup A_4\cup A_6$. 

Suppose that $X_1 \setminus (A_1 \cup A_3 \cup A_5 )= \emptyset$.
Then w.l.o.g.\ $|A_1|\geq 2$, and hence for a node $a_1 \in A_1$, $\{
a_1 \} \cup A_2 \cup A_6$ is a star cutset of $G$, a contradiction.
Therefore~(i) holds.

Let $C$ be a connected component of
$G[X_1 \setminus (A_1 \cup A_3 \cup A_5 )]$ and suppose that no node of $A_1$ has a neighbor
in $C$. Then for a node $a_4 \in A_4$, $\{ a_4 \} \cup A_3 \cup A_5$ is a star cutset of $G$ separating
$C$ from the rest, a contradiction.
 Therefore by symmetry,~(ii) holds.

 If $G$ is 4-hole-free then clearly $|A_i|=1$ for every $i\in\{ 1, \ldots ,6\}$, and if
 $\Delta (G) \leq 3$ then the same holds by~(i) and~(ii),
 therefore, (iii) holds.

 Since $G_1$ and $G_2$ are induced subgraphs of $G$,~(iv) holds.

 To prove (v) assume $G$ is 4-hole-free and w.l.o.g.\ $G_1$ has a star cutset $(x,R)$.
 Let $a_2,a_4,a_6$ be the marker nodes of $G_1$.
 By~(ii), $x \not\in \{ a_2,a_4,a_6 \}$. If $x \in A_1$, then $(x,R \cup A_2 \cup A_6)$ is a star cutset
 of $G$, a contradiction. Therefore by symmetry, $x \in X_1 \setminus (A_1 \cup A_3 \cup A_5)$.
 Since $G$ is 4-hole-free $R$ may contain nodes from at most one of the sets $A_1,A_3,A_5$,
 and hence $a_2,a_4,a_6$ are all contained in the same connected component of
 $G_1 \setminus (\{ x \} \cup R)$. It follows that $(x,R)$ is also a star cutset of $G$, a contradiction.
 Therefore (v) holds.
\end{proof}

We observe that property (v) above is not true in general for balanceable graphs. On the other hand, it is true for subcubic balanceable graphs. Since we will use a different technique to prove the main result for subcubic balanceable graphs than the one we will use for linear balanceable graphs, we will not need this result.

\medskip

A 2-join $(X_1,X_2)$ of $G$ is a {\em minimally-sided 2-join} if for some $i\in \{ 1,2 \}$ the following holds: for every 2-join $(X_1',X_2')$ of $G$, neither $X_1'\subsetneq X_i$ nor $X_2'\subsetneq X_i$. In this case $X_i$ is a {\em minimal side} of this minimally-sided 2-join.

\begin{lemma}[Trotignon and Vu\v{s}kovi\'c~\cite{nicolas.kristina:two}]\label{extreme}
Let $G$ be a graph with no star cutset. Let
$(X_1,X_2,A_1,A_2,B_1,B_2)$ be a split of a minimally-sided 2-join of
$G$ with $X_1$ being a minimal side, and let $G_1$ and $G_2$ be the
corresponding blocks of decomposition. Then the following hold:
\begin{enumerate}
\item $|A_1|\geq 2$, $|B_1|\geq 2$, and in particular all the vertices of $A_2\cup B_2$ are of degree at least 3.
\item If $G_1$ and $G_2$ do not have star cutsets, then $G_1$ has no 2-join.
\end{enumerate}
\end{lemma}

A partition $(X_1,X_2)$ of $V(G)$ is a {\em $\{ 2,6 \}$-join} if it is a 2-join or a 6-join of $G$.
It is a {\em minimally-sided $\{ 2,6 \}$-join} if for some $i\in \{ 1,2 \}$ the following holds:
for every $\{ 2,6 \}$-join $(X_1',X_2')$ of $G$, neither $X_1' \subsetneq X_i$
nor $X_2' \subsetneq X_i$. In this case $X_i$ is a {\em minimal side} of this minimally-sided
$\{ 2,6\}$-join.

\begin{lemma}\label{6jl2}
Let $G$ be a 4-hole-free bipartite graph.
Let $(X_1,X_2)$ be a minimally-sided $\{ 2,6 \}$-join of $G$, with $X_1$ being a minimal side.
If $G$ has no star cutset, then the block of decomposition $G_1$ has no $\{ 2,6 \}$-join.
\end{lemma}

\begin{proof}
  Assume the contrary, and let $(X_1',X_2')$ be a $\{ 2,6 \}$-join of $G_1$. We
  now consider the following cases.
  \\
  \\
  {\bf Case 1:} $(X_1,X_2)$ is a 2-join of $G$.
  \\
  By Lemmas~\ref{l:2j}~(ii) and~\ref{extreme}~(ii), $(X_1',X_2')$ is a
  6-join of $G_1$, say with split $(X_1',X_2',A_1', \ldots ,A_6')$.  Let
  $P_2$ be the marker path of $G_1$. By Lemma~\ref{6jl1}~(iii), we may
  assume w.l.o.g.\ that $V(P_2) \subseteq X_2'$. If $V(P_2) \subseteq
  X_2' \setminus (A_2' \cup A_4'\cup A_6')$, then clearly
  $(X_1',(X_2'\setminus V(P_2)) \cup X_2)$ is a 6-join of $G$ that
  contradicts the choice of $(X_1,X_2)$. So $V(P_2)\cap (A_2'\cup A_4'
  \cup A_6')\neq \emptyset$.  By Lemma \ref{6jl1}~(ii), we may assume
  w.l.o.g.\ that $V(P_2)\cap (A_4'\cup A_6')=\emptyset$.  But then
  $(X_1',(X_2'\setminus V(P_2)) \cup X_2, A_1',A_2,A_3',A_4',A_5',A_6')$
  is a split of a 6-join of $G$ that contradicts the choice of
  $(X_1,X_2)$.
  \\
  \\
  {\bf Case 2:} $(X_1,X_2)$ is a 6-join of $G$.
  \\
  Let $(X_1,X_2,A_1, \ldots ,A_6)$ be the split of this 6-join, and
  let $a_2,a_4,a_6$ be the marker nodes of $G_1$. We now consider the
  following two cases.
  \\
  \\
  {\bf Case 2.1:} $(X'_1,X'_2)$ is a 6-join of $G_1$.
  \\
  Let $(X_1',X_2',A_1', \ldots ,A_6')$ be the split of this 6-join. By
  Lemma~\ref{6jl1}~(iii) we may assume w.l.o.g.\ that $\{ a_2,a_4,a_6
  \} \subseteq X_2'\setminus (A_2' \cup A_4' \cup A_6' )$.  But then
  $(X_1',X_2' \cup X_2)$ is a 6-join of $G$ that contradicts the
  choice of $(X_1,X_2)$.
  \\
  \\
  {\bf Case 2.2:} $(X'_1,X'_2)$ is a 2-join of $G_1$.
  \\
  Let $(X_1',X_2',A_1',A_2',B_1',B_2')$ be the split of this 2-join.
  By Lemma~\ref{6jl1}~(iii), let $A_1=\{ a_1\}$, $A_3=\{ a_3\}$ and
  $A_5=\{ a_5\}$, and let $H$ be the 6-hole induced by $\{ a_1, \ldots
  ,a_6 \}$.  First suppose that both $X_1'\setminus (A_1' \cup B_1')$
  and $X_2'\setminus (A_2' \cup B_2')$ contain a node of $H$. Then
  w.l.o.g.\ we may assume that $a_2\in X_2'\setminus (A_2'\cup B_2')$,
  $a_4 \in B_1'$ and $a_6 \in A_1'$. Since nodes $a_2,a_4$ and $a_6$
  are all of degree~2 in $G_1$, it follows that $A_2'=\{ a_1 \}$ and
  $B_2'=\{ a_3\}$, and hence by Lemma~\ref{6jl1}~(iii) $(a_2, \{
  a_1,a_3 \})$ is a star cutset of $G$, a contradiction.

So we may assume w.l.o.g.\ that $(X_2'\setminus (A_2' \cup B_2')) \cap V(H)=\emptyset$.
By Lemma~\ref{l1}~(ii) and since $a_2,a_4,a_6$ are all of degree~2 in $G_1$, it follows that in
fact w.l.o.g.\ we may assume that $V(H) \cap X_2'\subseteq A_2'$.
By Lemma~\ref{l1}~(ii), every node of $A_2'$ has a neighbor in $X_2'$, and hence (since $a_2,a_4,a_6$ are all of degree~2 in $G_1$) $\{ a_2,a_4,a_6\} \subseteq X_1'$. But then $(X_1' \cup X_2,X_2')$ is a 2-join of $G$ that contradicts the choice of $(X_1,X_2)$.
\end{proof}

\section{Linear  balanceable graphs}\label{squarefree}

A {\em double star cutset} of a connected graph $G$ is a set $S$ of vertices such that $G\setminus S$ is disconnected and $S$ contains two adjacent vertices $u$ and $v$ such that every vertex of $S$ is adjacent to at least one of $u$ or $v$. Note that a star cutset is either a double star cutset or a cut vertex. If $U=(N(u)\cap S)\setminus \{v\}$ and $V=(N(v)\cap S)\setminus \{u\}$, then this double star cutset is denoted by $(u,v,U,V)$.
Note that if $G$ is a 4-hole-free bipartite graph, $U\cup V$ induce a stable set and $U\cap V=\emptyset$.

Let $C_i$, for $i=1, 2$, be a partition of the vertex set
$V(G\setminus S)$, such that there are no edges between vertices of
$C_1$ and $C_2$. Then $G_i=G[S\cup V(C_i)]$, $i=1,2$, are {\em blocks
  of decomposition} with respect to this double star cutset.

A double star cutset of a 2-connected graph $G$ with blocks of
decompositions $G_1$ and $G_2$ is a {\em minimally-sided double star
  cutset} if for some $i\in \{ 1,2 \}$ the following holds: for every
double star cutset of $G$ with blocks of decompositions $G_1'$ and
$G_2'$ neither $V(G_1')\subsetneq V(G_i)$ nor $V(G_2')\subsetneq
V(G_i)$. In this case $G_i$ is a {\em minimal side} of this
minimally-sided double star cutset.

\begin{lemma}\label{extremeStar}
Let $G$ be a 2-connected 4-hole-free bipartite graph that has a star cutset.
Let $G_i$, for some $i\in\{1,2\}$ be a minimal side of a minimally-sided double star cutset of $G$. Then $G_i$ does not have a star cutset.
\end{lemma}

\begin{proof}
Let $(u,v,U,V)$ be a minimally-sided double star cutset,
let $G_1$ be its minimal side, and let
$S=\{u ,v\} \cup U\cup V$.
Observe that every vertex of $U\cup V$ has a neighbor in $G_1\setminus S$.
In particular, $G_1$ is 2-connected.
Let us assume by way of contradiction that $(x,R)$ is a star cutset of $G_1$. Since $G_1$ is 2-connected, $R\neq\emptyset$.

\medskip

\noindent{\bf Case 1:} $x\not\in S$.
\medskip

\noindent  Since $G$ is 4-hole-free and bipartite, $x$ has at most one neighbor in $S$. If $R \cap \{u,v\}=\emptyset$, then  vertices of $S\setminus R$ are in the same connected component
of $G_1\setminus (\{ x \} \cup R)$,
and therefore $(x,y,R\setminus\{y\},\emptyset)$, for a vertex $y\in R$, is a double star cutset of $G$ that contradicts the minimality of $G_1$. So w.l.o.g.\  $u\in R$.
Let $C$ be a connected component of $G_1\setminus (\{ x \} \cup R)$ that does not contain a node
of $\{ v \} \cup V$. If $V(C)\setminus U\neq \emptyset$, then $(x,u,R\setminus \{ u \},U)$ is a double star cutset of $G$ that contradicts the minimality of $G_1$. So $V(C)\setminus U=\emptyset$.
But then some vertex $u' \in U$ is of degree~1 in $G_1$ (since $G_1$ is 4-hole-free and bipartite),
contradicting the fact that $G_1$ is 2-connected.

\medskip

\noindent{\bf Case 2:} $x\in S$.
\medskip

\noindent First, let us assume that $x\in\{u,v\}$, say $x=u$. Since $G$ is 4-hole-free and bipartite, every connected component of $G_1\setminus(\{x\}\cup R)$ that contains a vertex from $U$ or a vertex from $V$ contains a vertex from $G_1\setminus S$. Therefore, $(x,v,(U\cup R)\setminus\{v\},V)$ is a double star cutset of $G$ that contradicts the minimality of $G_1$.
So, $x\in U\cup V$, and w.l.o.g.\ we may assume that $x\in U$.
Then the nodes of $\{ v \} \cup V$ are all contained in the same connected component of $G_1 \setminus (\{ x \} \cup R)$.
Again, since $G$ is 4-hole-free and bipartite, every connected component of $G_1\setminus(\{x\}\cup R)$ that contains a vertex from $U$ contains a vertex from $G_1\setminus S$. Therefore, $(x,u,R\setminus\{u\}, U\setminus\{x\})$ is a double star cutset of $G$ that contradicts the minimality of $G_1$.
\end{proof}

Our main result about linear balanceable graphs is the following.

\begin{theorem}\label{deg2-square}
If $G$ is a linear balanceable graph on at least two vertices, then  $G$ contains at least two vertices of degree at most 2.
\end{theorem}

\begin{proof}
We prove the theorem by induction on  $|V(G)|$. If $|V(G)|=2$, then the theorem trivially holds. So, let $G$ be a linear balanceable graph such that $|V(G)|>2$.
We may assume that $G$ is connected, else we are done by induction.

Let $u$ be a cut vertex of $G$, and let $\{C_1,C_2\}$ be a partition of $V(G)\setminus\{u\}$, such that there are no edges between vertices of $C_1$ and $C_2$. Then, by induction
applied to graphs $G[C_i \cup \{ u \} ]$ for $i=1,2$, there is a vertex  $c_i\in C_i\setminus \{u\}$, for $i=1,2$, that is of degree at most 2 in $G[C_i\cup \{ u \} ]$. But then $c_1$ and $c_2$ are also of degree at most 2 in $G$.
So, we may assume that $G$ is 2-connected.

Now suppose that $G$ admits a star cutset. By Lemma~\ref{extremeStar},
there is a double star cutset $(u,v,U,V)$  of $G$, such that a block
of decomposition w.r.t.\ this cutset, say $G'$, has no star cutset. Let $S=\{u ,v\} \cup U\cup V$ and note that all vertices from $U$ and $V$ have a neighbor in $G' \sm S$. By Theorem~\ref{dt} $G'$ is basic or has a $\{ 2,6 \}$-join.

\medskip

\noindent {\bf Case 1:} $G'$ is basic.

\noindent  Let $(X,Y)$ be a bipartition of $G'$ such that all vertices of $Y$ are of degree~2.
Vertices $u$ and $v$ are adjacent, so we may assume w.l.o.g.\ that $\{v\} \cup U \subseteq Y$ and $\{ u \} \cup V \subseteq X$. In particular,  $|V|\leq 1$.

Suppose $V = \{v'\}$.  All the neighbors of $v'$ in $G' \sm S$ are of
degree~2 in $G'$ and in $G$, so we may assume that $v'$ has a unique
neighbor $w$ in $G' \sm S$. Let $w'$ be the unique neighbor of $w$ in
$G' \setminus v'$.  Since $G'$ is 4-hole-free and bipartite, $w'\in V(G')\sm
S$. If $w'$ is of degree~2 in $G'$ (and hence in $G$), then $w'$ and
$w$ are the desired two vertices. So we may assume that $w'$ has at
least three neighbors in $G'$. But then, since $G'$ is 4-hole-free
and bipartite, $w'$ must have a neighbor $w''\in V(G')\sm (S \cup \{ w
\} )$, and hence $w$ and $w''$ are the desired two vertices.

Now suppose that $V=\emptyset$ and let $v'$ be the neighbor of $v$ in
$V(G')\setminus S$. Since $G$ is 4-hole-free and bipartite, $v'$ has
no neighbors in $U\cup \{ u \}$.  So, either $\deg_{G'}(v')\geq 3$, in
which case $v'$ has at least two neighbors in $V(G')\setminus S$ of
degree~2 in $G'$, and hence in $G$, or $\deg_{G'}(v')=2$, in which
case $v'$ and the neighbor of $v'$ in $V(G')\setminus S$ are both of
degree~2 in $G'$, and hence in $G$.  Therefore $G$ has at least two
vertices of degree~2.

\medskip
\noindent{\bf Case 2:} $G'$ has a $\{ 2,6 \}$-join.

\noindent
Let $(X_1',X_2')$ be a $\{ 2,6\}$-join of $G'$. W.l.o.g.\ we may assume that $|X_1'\cap \{ u,v\} |\leq 1$.
Let $(X_1,X_2)$ be a minimally-sided $\{ 2,6\}$-join of $G'$ such that $X_1\subseteq X_1'$, and
let $G_1$ be the corresponding block of decomposition. Clearly $G_1$ is 4-hole-free and
$|X_1 \cap \{ u,v \} |\leq 1$. By Lemmas~\ref{l:2j} and~\ref{6jl1}, $G_1$ is linear balanceable and has
no star cutset. By Lemma~\ref{6jl2}, $G_1$ has no $\{ 2,6\}$-join, and hence by Theorem~\ref{dt},
$G_1$ is basic. We now consider the following two cases.

\medskip
\noindent{\bf Case 2.1:} $(X_1,X_2)$ is a 6-join of $G'$.

\noindent
Let $(X_1,X_2,A_1, \dots ,A_6)$ be the split of this 6-join. By Lemma~\ref{6jl1}, $A_1=\{ a_1\}$,
$A_3=\{ a_3 \}$, $A_5=\{ a_5 \}$, and all these nodes are of degree at least 3 in $G_1$.
Since $G_1$ is 4-hole-free, nodes $a_1,a_3,a_5$ do not have common neighbors in $X_1$.
Since $|X_1 \cap \{ u,v \} |\leq 1$, we may assume w.l.o.g.\ that $(X_1 \sm \{ a_1 \} ) \cap \{ u,v \} =\emptyset$. Let $a_3'$ (resp.\ $a_5'$) be a neighbor of $a_3$ (resp.\ $a_5$)  in $X_1$.
Then $a_3'\neq a_5'$ and $\{ a_3',a_5' \} \cap S=\emptyset$. Since $G_1$ is basic, $a_3'$ and $a_5'$ are of degree~2 in $G_1$, and hence in $G'$. Since $\{ a_3',a_5'\} \cap S=\emptyset$, they are also of degree~2 in $G$.

\medskip
\noindent{\bf Case 2.2:} $(X_1,X_2)$ is a 2-join of $G'$.

\noindent
Let $(X_1,X_2,A_1,A_2,B_1,B_2)$ be the split of this 2-join, and let $P_2$ be the marker path of
$G_1$. By Lemma~\ref{extreme}, $|A_1|\geq 2$, $|B_1|\geq 2$ and the ends of $P_2$ are
of degree at least 3 in $G_1$. Since $G_1$ is basic, it follows that the nodes of $A_1 \cup B_1$ are
all of degree~2 in $G_1$, and on the same side of bipartition of $G_1$, and hence of $G'$ as well.
In particular, it is not possible that both $u$ and $v$ are in $A_2 \cup B_2$.
Since $G'$ is 4-hole-free and bipartite, it follows that $|A_2|=|B_2|=1$, and hence the nodes of
$A_1\cup B_1$ are of degree~2 in $G'$. Since $|X_1\cap \{ u,v \} |\leq 1$, w.l.o.g.\ $B_1 \cap S=\emptyset$, and hence the nodes of $B_1$ are also of degree~2 in $G$.

\medskip

So, we may assume that $G$ does not admit a star cutset. Thus, by Theorem~\ref{dt} $G$ is basic or has a $\{ 2,6\}$-join.
So the theorem holds by the same proof as in Cases 1 and 2 above.
\end{proof}

\begin{corollary}\label{cr-square}
  Let $G$ be a linear balanceable graph that has at least one
  edge. Then there is an edge of $G$ that is not the unique chord of a
  cycle.
\end{corollary}
\begin{proof}
Follows immediately from Theorem~\ref{deg2-square} since an edge incident to a degree~2 vertex cannot be the unique chord of a cycle.
\end{proof}

\section{Subcubic balanceable graphs}\label{cubic}

A {\em branch vertex} is a vertex of degree at least 3. A {\em branch}
is a path connecting two branch vertices and containing no other
branch vertices. Two branches are {\em non incident} if the sets of
ends of the corresponding paths are disjoint. Note that a 2-connected
graph that is not a cycle is edgewise partitioned into its branches.
A pair of vertices $(u,v)$ of $G$ is a {\em pair of twins} in $G$ if
$N(u)=N(v)$ and $|N(u)|\geq 3$. Note that a cubic bipartite graph has
a pair of twins if and only if it contains a $K_{2,3}$ as a
subgraph. Note that $R_{10}$ does not have a pair of twins.

Our main result on subcubic balanceable graphs is the following theorem.

\begin{theorem}\label{main}
  Let $G$ be a 2-connected balanceable bipartite graph with
  $\Delta(G)\leq 3$. If $G$ is not equal to $R_{10}$ and has at least
  three branch vertices, then one of the following holds:
\begin{enumerate}[(i)]
\item $G$ has two vertices of degree~2 that are in non incident branches.
\item $G$ has a pair of twins and a vertex of degree~2.
\item $G$ has two disjoint pairs of twins.
\end{enumerate}
\end{theorem}

In the previous theorem, if $G$ has at least three branch vertices,
then it has in fact at least four branch vertices (because 2-connected
graphs have no vertex of degree~1).

\vspace{2ex}

The following lemma settles the case in which $G$ does not admit a star
cutset nor a 6-join. We treat this case separately because it does not
need induction.

\begin{lemma}\label{l2}
  Let $G$ be a 2-connected balanceable bipartite graph with $\Delta
  (G)\leq 3$, that is not equal to $R_{10}$ and has at least three
  branch vertices. If $G$ does not have a star cutset nor a 6-join,
  then $G$ has two vertices of degree~2 that are in non incident
  branches.
\end{lemma}
\begin{proof}
  By Theorem~\ref{dt}, $G$ is either basic or has a 2-join, so we
  consider the following two cases.  Note that every vertex of $G$ is
  of degree at least 2.  \medskip

\noindent{\bf Case 1:} $G$ is basic.
\\
Since $G$ is basic, no two branch vertices are adjacent, and hence every branch of $G$ contains
a vertex of degree~2.
Let $a,b,c$ be distinct vertices of degree~3, such that there is a branch from $a$ to $b$. There are three branches in $G$ with end  $c$. If one of the other ends of these branches is not $a$ or $b$, the proof is complete. So we may assume w.l.o.g.\  that we have two branches between $a$ and $c$ and one branch between $b$ and $c$. But then there is a branch from $b$ with an end not in $\{a,c\}$,
and hence the result follows.
\medskip

\noindent{\bf Case 2:} $G$ has a 2-join.
\\
Let $(X_1,X_2,A_1,A_2,B_1,B_2)$ be a split of a minimally-sided 2-join
of $G$ with $X_1$ being a minimal side. Let $G_1$ be the corresponding
block of decomposition. By Lemma~\ref{l:2j}, $G_1$ is balanceable and
it does not have a star cutset nor a 6-join.  By Lemmas~\ref{extreme}
and~\ref{l1}, $G_1$ has no 2-join, $|A_1|=|B_1|=2$, and all vertices of
$A_2 \cup B_2$ are of degree~3.  So by Theorem~\ref{dt} $G_1$ is
basic.

\medskip

\noindent{\bf Claim:} {\em $X_1 \setminus (A_1 \cup B_1)$ contains a vertex of degree~2.}
\\
\noindent{\em Proof of Claim:}
Assume not. Let $(X,Y)$ be a bipartition of $G_1$ such that all vertices of $X$ are of degree~2. Let $a_2, \ldots ,b_2$ be the marker path of $G_1$, with $a_2$ complete to $A_1$ and
$b_2$ complete to $B_1$. Then $a_2$ and $b_2$ are in $Y$ and hence $A_1 \cup B_1 \subseteq X$.
In particular, there are no edges in $G[A_1 \cup B_1]$.
So by Lemma~\ref{l1}, $X_1 \setminus (A_1 \cup B_1)$ is not empty.
By our assumption $X_1 \setminus (A_1 \cup B_1)\subseteq Y$. So for every $u \in X_1 \setminus (A_1 \cup B_1)$, $N(u) \subseteq A_1 \cup B_1$. But then since $|A_1\cup B_1|=4$, $|N(u)|\leq 3$ and the fact that each vertex of $A_1 \cup B_1$ is of degree~2 in $G_1$, we have a contradiction.
This completes the proof of the claim.

\medskip

By the claim let $c_1\in X_1 \setminus (A_1 \cup B_1)$ be of degree~2
(in $G_1$, and hence in $G$ as well).  Let
$(X_1',X_2',A_1',A_2',B_1',B_2')$ be a split of a minimally-sided
2-join of $G$ with $X_2'$ being a minimal side and $X_2'\subseteq
X_2$.  Then, as before, $|A_2'|=|B_2'|=2$, and hence all the vertices
of $A_1'\cup B_1'$ are of degree~3.  By the claim, there is a vertex
$c_2\in X_2'\setminus (A_2' \cup B_2')$ that is of degree~2 in $G$.

Since $|A_1| = |B_1| = |A'_2| = |B'_2| = 2$, we see that no branch of
$G$ may overlap the three following sets: $A_1 \cup B_1$, $X_1
\setminus (A_1 \cup B_1)$ and $A'_2 \cup B'_2$ (resp.\ $A'_2 \cup
B'_2$, $X'_2 \setminus (A'_2 \cup B'_2)$ and $A_1 \cup B_1$).  It
follows that $c_1$ and $c_2$ are in non incident branches.
\end{proof}

\medskip

\noindent {\em Proof of Theorem~\ref{main}:} We proceed by induction on $|V(G)|$.
If $|V(G)| = 1$, then the theorem is vacuously true.  By
Theorem~\ref{dt} and Lemma~\ref{l2}, we may assume that $G$ has a
star cutset or a 6-join.

\medskip

\medskip

\noindent{\bf Proof when $G$ has a star cutset.}
\\
Let $(x,R)$ be a star cutset of $G$ such that $|R|$ is minimum.  Since
$G$ is 2-connected, $|R|\geq 1$, and by the choice of $(x,R)$ and since
$G$ is subcubic, every vertex of $R$ has neighbors in every connected
component of $G\setminus (\{x\}\cup R)$, every vertex of $R$ is of
degree~3 and $G\setminus (\{x\}\cup R)$ has exactly two connected
components, say $C_1$ and $C_2$. Let $G_i$ be the block of
decomposition w.r.t.\ this cutset that contains $C_i$, for
$i=1,2$. Note that every vertex of $R$ is of degree~2 in $G_i$. Note
also that both $G_1$, $G_2$ are 2-connected.
\medskip

\noindent{\bf Claim:} {\em If $x$ is of degree~2 in $G_i$, for
  some $i\in\{1,2\}$, then $C_i$ contains a vertex $u$ of degree~2, or
  a pair of twins.  Furthermore, if $G_i$ has at least two branch
  vertices, then $u$ can be chosen so that $x$ and $u$ are not in the
  same branch of $G_i$.}
\\
\noindent{\em Proof of Claim:}
If $G_i$ has no branch vertices, then $C_i$ contains a vertex of
degree~2. If $G_i$ has exactly two branch vertices, both are in $C_i$.
Since these vertices can have at most one branch of length 1
connecting them, there must be a branch between them that is fully contained in $C_i$ and is of length at least~2, and
therefore there is a vertex of degree~2 in $C_i$ that is
not in the same branch as $x$. If $G_i$ has at least 3 branch
vertices, then, by the induction hypothesis, $C_i$ contains a vertex
of degree~2 that is not in the same branch as $x$, or $C_i$ contains a
pair of twins. This completes the proof of Claim.  \medskip

We now consider  the following cases.

\medskip

\noindent{\bf Case 1:} $|R|=1$.
\\
\noindent
Note that since $G$ is 2-connected, $x$ has a neighbor in both $C_1$
and $C_2$, and in particular, $x$ is of degree~2 in both $G_1$ and
$G_2$. Since $G$ has at least three branch vertices, at least one of
$G_1$ or $G_2$ has at least two branch vertices, so, by Claim
applied for $i=1$ and $i=2$, $G$ satisfies the theorem.  \medskip

\noindent{\bf Case 2:} $|R| =2$.
\\
\noindent
Let $R=\{y_1,y_2\}$. Suppose that $\deg(x)=2$. Then at least one of
$G_1$ or $G_2$ has at least two branch vertices (since neither can
have exactly one), w.l.o.g.\ say $G_1$ does. By Claim applied to
$G_1$, there is a degree 2 vertex $u$ in $C_1$ that is not in the same
branch of $G_1$ as $x$. Since $y_1$ and $y_2$ have degree~3 in $G$,
$x$ and $u$ are degree 2 vertices of $G$ that are contained in non
incident branches of $G$, a contradiction.  So $\deg(x)=3$, and
w.l.o.g.\ $x$ has a neighbor in $C_1$ and does not in $C_2$.  If $G_1$
has exactly two branch vertices and they are adjacent, then for a
shortest path $P$ from $y_1$ to $y_2$ in $G_2\setminus\{x\}$, the set
$V(G_1)\cup V(P)$ induces an odd wheel with centre $x$, contradicting
Theorem \ref{balanceable-truemper}. So, if $G_1$ has exactly two
branch vertices, then there is a vertex of degree~2 in $G_1$ in a
branch that does not contain $y_1$ nor $y_2$, and therefore, by Claim
applied to $G_2$, $G$ satisfies the theorem, a contradiction.  So
$G_1$ must have at least three branch vertices, and hence by induction
hypothesis, $G_1$ has a pair of twins or a vertex of degree~2 in a
branch that has both of its ends in $C_1$.  But then by Claim applied
to $G_2$, $G$ satisfies the theorem.  \medskip

\noindent{\bf Case 3:} $|R|=3$.
\\
\noindent
Let $R=\{y_1,y_2,y_3\}$.  First, let us suppose that both $G_1$ and
$G_2$ have exactly two branch vertices, and that $v_i$ is a branch
vertex of $G_i$ different from $x$, for $i=1,2$. If $G_i$, for
$i=1,2$, does not have a vertex of degree~2 other than $y_j$, for
$j=1,2,3$, then $G$ is a $K_{3,3}$, and hence it satisfies
(iii) of the theorem. So, we may assume that there is a vertex
of degree~2 (in $G$) in a branch of $G_1$ containing $y_1$. If
$y_2v_2$ or $y_3v_2$ is not an edge, then $G$ satisfies (i) of
the theorem, so we may assume that $y_2v_2$ and $y_3v_2$ are edges. If
$y_1v_2$ is also an edge, then $x$ and $v_2$ form a pair of twins, and
therefore $G$ satisfies~(ii) of the theorem. When $y_1v_2$ is not an
edge, then by symmetry $v_1y_2$ and $v_1y_3$ are edges. But then $y_2$
and $y_3$ form a pair of twins, and therefore $G$ satisfies
(ii) of the theorem.

Observe that if $G_i$ has at least three branch vertices, then, by
induction hypothesis, there is a vertex $u_i$ of degree~2 in a branch of
$G_i$ not having $x$ as its end, or $G_i$ has a pair of twins that
does not contain $x$ (since $G_i$ has at least three branch vertices).
So if both $G_1$ and $G_2$ have at least three branch vertices, then
the theorem holds.  Therefore we may assume that $G_1$ has at least
three and $G_2$ exactly two branch vertices. If $G_2$ has a vertex2
$u_2$ of degree~2 not in $\{y_1,y_2,y_3\}$, then $G$ satisfies~(i)
or~(ii) of the theorem. So we may assume that the only vertices of
$G_2$ of degree~2 are $y_1,y_2$ and $y_3$, and therefore $x$ and the
other branch vertex of $G_2$ form a pair of twins, hence $G$ satisfies
(ii) or~(iii). This completes the proof when $G$ has a star
cutset.  \medskip

\noindent
{\bf Proof when $G$ has a 6-join.}

We may assume that $G$ has no star cutset.  In particular, $G$ does
not contain a pair of twins (for if $u,v$ is a pair of twins of $G$, since $G$
has at least three branch vertices, $V(G)\setminus (N(u)\cup \{ u,v \} )\neq \emptyset$, and hence $N(u)\cup \{ u \}$
is a star cutset).  Let $(X_1,X_2,A_1,A_2,A_3,A_4,A_5,A_6)$
be a split of a 6-join of $G$ and let $A=\cup_{i=1}^{6}A_i$.  By
Lemma~\ref{6jl1}~(iii), $|A_i|=1$ for every $i \in \{ 1, \ldots ,6\}$
and all nodes of $A$ are of degree~3 in $G$.  It follows that both
blocks of decomposition $G_1$ and $G_2$ have at least three branch
vertices. By the choice of $G$, each of them has a vertex of degree~2
not in $A$, and hence $G$ satisfies (i) of the theorem.  This
completes the proof.\hfill$\Box$

\medskip

As a consequence of Theorem~\ref{main} we have the following
corollary, a special case of which was conjectured in~\cite{spiga}.

\begin{corollary}\label{twins}
If $G$ is a cubic balanceable graph that is not $R_{10}$, then $G$ has a pair of twins none of whose neighbors is a cut vertex of $G$.
\end{corollary}

\begin{proof}
Let $G'$ be an end block of $G$.
Then $G'$
has at most one vertex of degree~2, and all the other vertices of degree~3.
If $G'$ does not have a vertex of degree~2, then let $G''=G'$, and otherwise let $G''$ be the graph obtained from $G'$ by subdividing twice an edge incident to the degree~2 vertex. Clearly $G''$ is 2-connected
balanceable and not equal to $R_{10}$. Note that $G''$ has at most one branch of length greater than 1.
By Theorem~\ref{main} $G''$ has a pair of twins $\{ u_1,u_2 \}$.
Note that none of the neighbors of $u_1$ and $u_2$ in $G''$ can be of degree~2 in $G''$, and hence
$\{ u_1,u_2 \}$ is the desired pair of twins of $G$.
\end{proof}

As was noticed in~\cite{spiga} (for the special case of cubic
balanced graphs), Corollary~\ref{twins} implies the following.

\begin{corollary}\label{cr-cubic}
Let $G$ be a cubic balanceable graph. Then the following hold:
\begin{enumerate}[(i)]
\item $G$ has girth four.
\item If $G\neq R_{10}$ then $G$  contains an edge that is not the unique chord of a cycle.
\item $G$ is not planar.
\end{enumerate}
\end{corollary}

\begin{proof}
  It is easy to see that if $G=R_{10}$ then~(i) and~(iii) hold. So we
  may assume that $G \neq R_{10}$.  By Corollary~\ref{twins}, let
  $\{u_1,u_2\}$ be a pair of twins of $G$, and $\{v_1,v_2,v_3\}$ the
  set of neighbors of $u_1$ and $u_2$.  Then $u_1v_1u_2v_2$ is a cycle
  of length 4, and hence~(i) holds.  Suppose that $u_1v_1$ is a unique
  chord of a cycle $C$ in $G$.  Then all neighbors of $u_1$ and $v_1$
  belong to $C$, and in particular, $u_2$ belongs to $C$ and has three
  neighbors in $C$, a contradiction. Hence~(ii) holds.

  By Corollary~\ref{twins} we may assume that none of $v_1,v_2,v_3$ is
  a cut vertex of $G$. So there is a connected component $C$ of $G
  \setminus \{ u_1,u_2,v_1,v_2,v_3 \}$ such that all of $v_1,v_2,v_3$
  have a neighbor in $C$. Let $C'$ be a minimal induced subgraph of
  $C$ that is connected and all of $v_1,v_2,v_3$ have a neighbor in
  $C'$. Since $G$ is cubic, it is easy to see that $V(C') \cup \{
  u_1,u_2,v_1,v_2,v_3 \}$ induces a subdivision of
  $K_{3,3}$. Therefore, by Kuratowski's Theorem (see for
  example~\cite{bondy.murty:book}), $G$ is not planar.
 \end{proof}

\section*{Acknowledgement}

Thanks to Pablo Spiga for useful discussions.

\end{document}